\documentclass[12pt]{amsart}
\usepackage{amstext,amsthm,amsmath,amssymb,amscd}
\usepackage[utf8]{inputenc}
\usepackage[T1]{fontenc} 
\usepackage{pstricks,pstricks-add}
\usepackage{paralist}
\usepackage{color}
\definecolor{darkblue}{rgb}{.8,.15,.15}
\definecolor{darkgreen}{rgb}{0.15,.4,.5}
\usepackage[naturalnames,
            colorlinks,linkcolor=darkblue,
            citecolor=darkblue,
            filecolor=darkblue,urlcolor=darkgreen,
            anchorcolor=darkblue,menucolor=darkblue,
            bookmarksopen=false,bookmarks=false
            ]{hyperref} 
\newcommand\NN{{\mathbb N}} 
\newcommand\RR{{\mathbb R}} 
\newcommand\GG{{\mathbb G}}
\newcommand\FF{{\mathcal F}}
\newcommand\PP{{\mathcal P}}

\newcommand\x{{\mathbf x}}
\newcommand\e{{\mathbf e}}
\def\opn#1#2{\def#1{\operatorname{#2}}}
\opn\aff{aff}
\opn\conv{conv}
\opn\cone{cone}
\opn\rank{rank}
\opn\hom{hom}
\opn\maximum{max}
\opn\max{max}
\opn\minimum{min}
\opn\Pr{{\rm Pr}}
\opn\windmill{Wd}
\opn\ex{ex}
\opn\z{z}
  
\theoremstyle{plain}
\newtheorem{theorem}{Theorem}
\newtheorem{lemma}[theorem]{Lemma}
\newtheorem{corollary}[theorem]{Corollary}
\newtheorem{proposition}[theorem]{Proposition}

\theoremstyle{definition}
\newtheorem{definition}[theorem]{Definition}
\newtheorem{example}[theorem]{Example}

\newtheorem*{remark*}{Remark}

\usepackage{a4}

\begin{document}

\title{Extremal edge polytopes} 

\author{Tuan Tran}
\address{Institute of Mathematics, Freie Universität Berlin, Arnimallee 3, D-14195 Berlin, Germany}
\email{tuan@math.fu-berlin.de} 

\author{Günter M. Ziegler}
\address{Institute of Mathematics, Freie Universität Berlin, Arnimallee 2, D-14195 Berlin, Germany}
\email{ziegler@math.fu-berlin.de}

\thanks{The first author was supported by DFG within the Research Training Group ``Methods for Discrete Structures''; the research of
the second author was also funded by the European Research Council under the European Union's Seventh Framework Programme (FP7/2007-2013) / ERC Grant agreement no.~247029-SDModels.}

\begin{abstract}
The \emph{edge polytope} of a finite graph $G$ is the convex hull of the columns of its vertex-edge 
incidence matrix.  We study extremal problems for this class of polytopes. For $k \in \{2,3,5\}$, we determine the 
maximum number of vertices of $k$-neighborly edge polytopes up to a sublinear term.    
We also construct a family of edge polytopes with exponentially-many facets. 
\end{abstract}
\maketitle

\section{Introduction} 
 
 The main object of our investigation is a special class of $0/1$-polytopes (cf.~\cite{Z67}): 
 The \emph{edge polytope} $\PP(G)$ of a graph $G$ on the vertex set $\{1,2,\dots,n\}$
 is the polytope generated by all vectors $\e_i + \e_j$ such that $i$ is adjacent to $j$, where $\e_i$ and $\e_j$ stand for the $i$th and $j$th unit vectors of $\RR^n$. For example, the edge polytopes of trees are simplices, while the
 edge polytope of the complete graph~$K_n$ is the second hypersimplex $\Delta_{n-1}(2)$.
Thus the edge polytopes are the subpolytopes of the second hypersimplex. 
 A study of edge polytopes of general graphs was initiated by Ohsugi \& Hibi \cite{OH} and Villarreal \cite{Vil}, who 
 both provided the half-space description of these polytopes (cf.~Theorem~\ref{thm:the_facets}). 
 Dupont \& Villarreal \cite{DV} have recently connected this to the setting Rees Algebras in combinatorial commutative algebra. For further discussions of edge polytopes, see \cite{HLZ}, \cite{MHNOH}, \cite{OH1}, and~\cite{OH2}.
 
In this paper we demonstrate that edge polytopes form a rich family of $0/1$-polytopes with
interesting random and extremal properties.
In particular, we obtain 
edge polytopes with an exponential number of facets
 (see Theorem~\ref{thm:exponential_facets}) and $k$-neighborly $0/1$-polytopes with more than linearly
 many vertices for any $k \ge 2$ (Corollary~\ref{cor:turan_polytope}).
 On the other hand we will show that edge polytopes can be described and analyzed in terms of 
 parameters of the graphs they are based on (and thus are not as intractable as the same problems for 
 general $0/1$-polytopes \cite{KGSZ} \cite{Z67}). Thus  
we obtain structural overview concerning three main topics:
 \begin{compactenum}[(1)]
 \item a description of the low-dimensional faces of the polytope $\PP(G)$; 
 \item non-linear relations between the components of the $f$-vector of $\PP(G)$;
 \item the asymptotics of the maximal number of facets of $d$-dimensional edge polytopes for large~$d$.
 \end{compactenum}
Here are some remarks connected to this.
 To describe all low-dimensional faces of $\PP(G)$ we only need to consider ``small'' induced subgraphs of $G$.   
 The second topic  
 is closely related to  
 the problem of finding minimal density of a fixed bipartite graph in a dense graph.
 Concerning the third topic Gatzouras et al.~\cite{GGM}, improving on a breakthrough by B\'{a}r\'{a}ny \& P\'{o}r \cite{BP}, showed that there are random $0/1$-polytopes in $\RR^d$ with as many as $\big ( \frac{cd}{\log^2{d}} \big )^{d/2}$ facets (or more), where $c>0$ is an absolute constant. 
 The situation for $d$-dimensional random edge polytopes, however, where we have only a polynomial (quadratic) number of potential vertices, turns out to be quite different from that of general $0/1$-polytopes.

 The paper is divided into five sections.
 In the next section we introduce the object of our investigation and determine the dimension of an arbitrary edge polytope.  
 A criterion for determining faces of edge polytopes is provided.

 In Section~\ref{sec:2} we compute the number of edges of $\PP(G)$ in terms of the number of vertices of $\PP(G)$, the number of $4$-cycles and the number of $4$-cliques in~$G$.
 The function $g(n):=\maximum{\{f_1(\PP(G)):  G \text{ has }n\text{ vertices} \}}$ is 
 in Theorem~\ref{Bound the number of edges}
 shown to be of order $\Theta(n^4)$. The lower bound for this is provided by random edge polytopes. 
 
 In Section~\ref{sec:3} we characterize $k$-neighborly edge polytopes for $k \geq 2$.
 We then obtain a tight upper bound on the number of vertices of these edge polytopes,
 by counting various types of walks in the graph.
 All edge polytopes which attain these bounds are pseudo-random in some sense.  
 
 In Section~\ref{sec:4} we use results of Ohsugi \& Hibi \cite{OH} to show that a $d$-dimensional edge polytope has at most $2^d + d$ facets.
 Inspired by Moon \& Moser \cite{MM}, we provide a construction for $d$-dimensional edge polytopes with 
 roughly $4^{d/3}$ facets.
 
 \section{Preliminaries}
 All graphs in this paper are finite, undirected, with no loops, no multiple edges and no isolated vertices. The vertex and edge sets of a graph $G$ are denoted by $V(G)$ and $E(G)$.  We write  
 $|G|$ for the number of vertices of $G$, and  
 $e(G)$ for the number of edges.
 We write $G[S]$ for the subgraph of $G$ induced by a set $S \subseteq V(G)$. Given two sets $S, T\subseteq V(G)$, not necessarily disjoint, we write $e_G(S,T)$ for the number of ordered pairs $(s,t)$ with $s \in S, t \in T$ and $st \in E(G)$. Given a non-empty subset $X \subseteq V(G)$, the {\em neighbor} set of $X$ in $G$ is $N_G(X):=\{v \in V(G)\setminus X: \ \textrm{$v$ is adjacent to some vertex in $X$}\}$. If the graph $G$ is clear from the context, we often write $N(X)$ instead of $N_G(X)$.
 
 The main object of our consideration is a special class of $0/1$-polytopes.
 
 \begin{definition}
 Let $G$ be a graph on the vertex set $\{1,2,\ldots,n\}:=[n]$.  
 The \emph{edge polytope} $\PP(G)$ of $G$ is the convex hull of all vectors $\e_i + \e_j$ such that $i$ is adjacent to $j$, where $\e_i$ and $\e_j$ denote the $i$th and $j$th unit vectors of $\RR^n$. 
 \end{definition}
 
 Thus the edge polytopes of $n$-vertex graphs correspond to the subpolytopes of the second hypersimplex of order $n$.
 
 \begin{example} 
 The \emph{second hypersimplex} of order $n$ is defined as
 \[
 \Delta_{n-1}(2):=\conv \{\e_i+\e_j: 1\le i < j \le n\}\subseteq \RR^{n}.
 \]
 It is the edge polytope of the complete graph $K_n$. 
 The second hypersimplex $\Delta_{n-1}(2)$ has dimension $n-1$
 (it is contained in $\{\x\in\RR^n:x_1+\dots+x_n=2\}$), $\binom n2$~vertices, and $2n$ facets if $n \ge 4$. 
 For example, $\Delta_3(2) \subseteq \RR^4$ is affinely equivalent to the regular octahedron.%
 \end{example}
 
 \begin{example}
 Let $C_n$ be the cycle $(1,2,\ldots,n{-}1,n)$. 
 If $n$ is odd, then $\e_1+\e_2,\e_2+\e_3$, $\ldots,\e_{n-1}+\e_n,\e_n+\e_1$ are affinely independent,
 so the edge polytope $\PP(C_n)$ is an $(n-1)$-simplex.  
 If $n$ is even, then the edge polytope $\PP(C_n)$ has dimension $n-2$:
 It is a sum of two $(\frac n2-1)$-simplices in $\{\x\in\RR^n:x_1+\dots+x_n=2,\sum_{i=1}^{n}(-1)^ix_i=0\}$.
 \end{example}
 
The definition of edge polytope can be rephrased in the following way.
Let $G$ be a graph. The \emph{incidence matrix} of $G$ is the matrix $A \in \{0,1\}^{V(G)\times E(G)}$ with
 \begin{equation*} a_{v,e}=
 \begin{cases}
 1 & \text{ if $v \in e$}, \\
 0 & \text{ otherwise.} 
 \end{cases}
 \end{equation*}
 The edge polytope $\PP(G)$ of $G$ is precisely the convex hull of the column vectors of the matrix $A$. 
 Hence $\PP(G)$ can be obtained by taking the intersection between the cone $\cone{(A)}$ and the hyperplane $\{\x \in \RR^n: x_1+\dots+x_n=2 \}$.  Thus
 \[
 \dim{(\PP(G))}= \dim{(\cone{(A)})}-1=\rank{(A)}-1.
 \]
 On the other hand, by \cite[Theorem 8.2.1]{GR}, $\rank{(A)}=n-c_0(G)$,
 where $c_0(G)$ is the number of connected bipartite components of $G$. 
 We have therefore determined $\dim{(\PP(G))}$. 
 
 \begin{lemma}[{Valencia \& Villarreal \cite[p.~57]{VV}}]\label{lem:dimension} 
 Let $G$ be a graph with $n$ vertices and $c_0(G)$ bipartite components.
 Then the dimension of the edge polytope $\PP(G)$ of $G$ equals $n-c_0(G)-1$. 
 \end{lemma}
 
 This result enables us to obtain a quadratic upper
  bound on the number of vertices of the polytope $\PP(G)$ in terms of its dimension.
 
 \begin{proposition} \label{prop:number_of_edges_and_dimension}
 Let $G$ be a finite graph, and let $d=\dim{(\PP(G))}+1$.  
 If $d \ge 4$, then $e(G) \le \binom d2$. 
 Equality holds if and only if $G$ is a complete graph with at least~$4$ vertices.
 \end{proposition} 

 For further investigation, we need a criterion for determining faces of an edge polytope. 
 From now on, the symbol $\e_{ij}$ is used to denote the vector $\e_i+\e_j$.
 
 We will use the following simple criterion.  
 
 \begin{lemma} \label{lem:face_criterion}
 Let $V \subset \RR^n$ be the vertex set of a polytope $P$ and let $U \subseteq V$.  
 Then $U$ is the vertex set of a face of $P$ if and only if 
 $\aff{(U)} \cap \conv{(V {\setminus} U)} = \emptyset$. 
 \end{lemma}
  
 For edge polytopes, this criterion can be reformulated as follows.
 
 \begin{lemma} \label{lem:local_property}
 Let $H$ be a subgraph of $G$ with the vertex set $U$. 
 Then $\PP(H)$ is a face of $\PP(G)$ if and only if 
 $\PP(H)$ is a face of $\PP(G[U])$.  
 \end{lemma}
 \begin{proof} The ``only if'' part follows from the fact that $\PP(G[U])$ is a subpolytope of $\PP(G)$ and $\PP(H) \subseteq \PP(G[U])$.
 For the ``if'' part, it is enough to show that $\PP(G[U])$ is a face of $\PP(G)$. 
 Observe that all points $\x \in \aff{\big \{\e_{ij}: \{i,j\} \in E(G[U])\big \}}$ satisfy $\x_i =0$ for all $i \notin U$. 
 Hence we have 
\[\aff{\big \{\e_{ij}: \{i,j\} \in E(G[U])\big \}} \cap \conv{\big \{\e_{ij}: \{i,j\} \in E(G) {\setminus} E(G[U]) \big \}} =\emptyset.\]
 So $\PP(G[U])$ is a face of $\PP(G)$, by Lemma~\ref{lem:face_criterion}.
 \end{proof}
 
 \section{Graphs of edge polytopes}\label{sec:2}
 
 The following simple result of Ohsugi and Hibi identifies the edges of $\PP(G)$. 
 
 \begin{lemma} [Ohsugi \& Hibi \cite{OH}] \label{lem:edge_set}
 The vertices $\e_{ij}$ and $\e_{k\ell}$ of~$\PP(G)$ form an edge if and only if
 \begin{compactenum}[\rm(i)]
 \item  the two edges $\{i,j\}$ and $\{k,\ell\}$ have a common vertex, or 
 \item $\{i,j\} \cap \{k,\ell\}=\emptyset$, $\{i,j\}$ and $\{k,\ell\}$ are not contained in any $4$-cycle of $G$.  
 \end{compactenum}
 \end{lemma}
 
\begin{pspicture}(-0.70,0.00)(12.00,6.60)
\pspolygon(7,6)(5,4)(7,2)(9,4)
\psline(6.60,3.20)(7,6)
\psline(6.60,3.20)(5,4)
\psline(6.60,3.20)(7,2)
\psline(6.60,3.20)(9,4)
\psline[linestyle=dashed](7.40,4.80)(7,6)
\psline[linestyle=dashed](7.40,4.80)(5,4)
\psline[linestyle=dashed](7.40,4.80)(7,2)
\psline[linestyle=dashed](7.40,4.80)(9,4)
\psline[linestyle=dotted](7,6)(7,2)
\psline[linestyle=dotted](5,4)(9,4)
\psline[linestyle=dotted](6.60,3.20)(7.40,4.80)
\put(6.80,6.15){12}
\put(4.45,3.90){13}
\put(6.80,1.55){34}
\put(9.10,3.90){24}
\put(6.15,2.85){14}
\put(7.45,4.85){23}
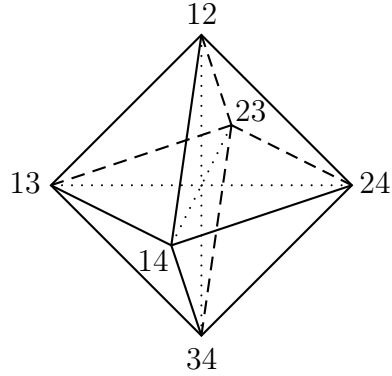
\put(2.30,0.60){Figure 1: The octahedron $\PP(K_4)$, a face of $\PP(K_n)$. }
\end{pspicture}

Using Lemma~\ref{lem:edge_set}, we can compute the number of edges of $\PP(G)$. 
For this, let $c_4(G)$ and~$k_4(G)$ be the numbers of copies of $C_4$ and of $K_4$ in $G$,
respectively.
As usual in polytope theory, we denote by $f_k$ the number of $k$-dimensional faces of a polytope.

\begin{proposition} \label{prop:number_of_edges}
If $\PP(G)$ be the edge polytope of a simple graph $G$, then $f_0(\PP(G))=e(G)$ and
\[
    f_1(\PP(G))=\tbinom{e(G)}2 -2c_4(G)+3k_4(G).
\]
\end{proposition}

\begin{proof}
Let $r_4(G)$ denote the number of pairs of disjoint edges of $G$ which are contained in some 
$4$-cycle of~$G$. 
Lemma~\ref{lem:edge_set} shows that $f_1=\binom{f_0}2-r_4(G)$.  
Thus Proposition~\ref{prop:number_of_edges} will be established if $r_4(G)-2c_4(G)+3k_4(G)=0$ holds. Now 
\[
    r_4(G)-2c_4(G)+3k_4(G)=\sum_{U \subseteq [n],\,|U|=4}\bigl( r_4(G[U])-2c_4(G[U])+3k_4(G[U])\bigl),
\]
so we can assume from the beginning that $|G|=4$.  
The rest is  
left to the reader.
\end{proof} 

We next give a sharp lower bound for $f_1(\PP(G))$ in terms of $f_0(\PP(G))$. 

\begin{theorem} \label{number of vertices and edges}
    If $f_0$ and $f_1$ denote the number of vertices resp.\ edges  
    of the edge polytope $\PP(G)$, then  
\[
    f_0^{3/2}-f_0 \le f_1.
\]  
Equality holds if and only if $G$ is a complete bipartite graph with equal size parts.%
\end{theorem}

\begin{proof}
Without restriction we can assume that $G$ is a connected graph on $n$ vertices. 
Let $\bar{d}$ be the average degree of $G$. 
Since $G$ is connected, it has at most one bipartite component. Thus we have $d:=\dim{(\PP(G))} = n-c_0(G)-1 \ge n-2$. 
Now we count the following set $S$ in two ways: $S$ is the set of incidence pairs $(v,e)$ where $v$ is a vertex of $\PP(G)$, and $e$ is an edge of $\PP(G)$.  
Here $2f_1 = |S| \ge df_0$, and hence $f_1 \ge \frac{df_0}{2}$.  
It follows that
\[\textstyle
  f_1 \ge \frac{df_0}{2} \ge \frac{(n-2)f_0}{2} = \frac{nf_0}{2}-f_0.
\] 
Next we let $T$ be the set of pairs $(u,\{v,w\})$ where $u$ is adjacent to $v$ and to $w$, with $v \neq w$. 
Since $\{u,v\}$ and $\{u,w\}$ are two edges of $G$ which have one common node, they form an edge of $\PP(G)$, by 
Lemma~\ref{lem:edge_set}. Therefore we have
\[\textstyle
    f_1 \ge |T| = \sum_{v \in V(G)} \binom{\deg(v)}2 \ge n \binom{\bar{d}}2 =\frac{2f_0^2}{n} -f_0.
\] 
Combining this inequality with the previous one, we get
\[\textstyle
f_1 \ge \frac{1}{2} \big \{ \frac{nf_0}{2} + \frac{2f_0^2}{n} \big \} -f_0 \ge f_0^{3/2}-f_0.
\] 
It is easy to check that the equality holds if and only if $G \cong K_{n/2,n/2}$.
\end{proof}

Let $\mathcal{P}(G)$ be the edge polytope of a graph $G$ and $f_1(\PP(G))$ be its number of edges.
How large can $f_1(\PP(G))$ be if the number of vertices of $G$ is fixed? 
Formally, we want to know the asymptotic behaviour of the number of edges of the polytope
in terms of the number $n$ of vertices of the graph,
\[
g(n):=\maximum{\{f_1(\PP(G)):  |G|=n\}}.
\]
For that we use the following bound on the number of copies of a fixed complete bipartite subgraph of a graph
of given density.
 
\begin{lemma} [{Alon \cite[Corollary 2.1]{Alon}}] \label{bipartite subgraphs} 
For every fixed $\varepsilon > 0$, any two fixed integers $s \ge t \ge 1$, 
and for any graph $G$ with $n$ vertices and at least~$\varepsilon n^2$ edges, the number of subgraphs of $G$
isomorphic to $K_{s,t}$ is at least
\[\textstyle
 (1+o(1))\binom ns \binom nt (2 \varepsilon)^{st}
\]
if $s>t$, and at least 
\[\textstyle
 (\tfrac{1}{2}+o(1)) \binom ns \binom nt (2 \varepsilon)^{st}
\]
for $s=t$, where the $o(1)$ terms tend to $0$ as n tends to infinity.
\end{lemma}
 
It is worth noting that the assertions of the above lemma are tight,
as shown by the random graph $G(n,2 \varepsilon)$ on $n$ labeled vertices in which each pair of vertices 
is an edge with probability $2 \varepsilon$. 
 
And now, as promised, we provide bounds for the function $g(n)$.  

\begin{theorem} \label{Bound the number of edges} 
For every integer $n \ge 6$, the function $g(n)=\max{\{f_1(\PP(G)): |G|=n\}}$ satisfies
\[
  \tfrac{1}{54}n^4 \le g(n) \le  ( \tfrac{1}{32} + o(1))  n^4.
\]
\end{theorem}

\begin{proof}~\\
(i) \emph{Lower bound} \\
For simplicity of notation, we write $\GG$ instead of $G(n,p)$. 
Define
\[
    p_1=\Pr{(\e_{12},\e_{13} \ \textrm{form an edge of} \ \PP(\GG))}
\]
and
\[
   p_2=\Pr{(\e_{12},\e_{34} \ \textrm{form an edge of} \ \PP(\GG))}.
\]
Lemma~\ref{lem:edge_set} shows that $\conv{\{\e_{12},\e_{13}}\}$ is an edge of $\PP(\GG)$ iff $\e_{12}, \e_{13} \in \GG$, and thus $p_1=p^2$.  
To compute $p_2$, note that $\{1,2\}$ and $\{3,4\}$ are not contained in any 4-cycle of $\GG$ if and only if not both $\{1,3\}, \{2,4\}$ and not both $\{1,4\},\{2,3\}$ belong to $\GG$ (see Figure~1).
From this, we get $p_2=p^2(1-p^2)^2$. 
Thus by linearity of expectation
\begin{align*}
Ef_1 &=\textstyle n \binom{n-1}2 p_1 + \tfrac12 \binom n2 \binom{n-2}2  p_2 \\
     &=\textstyle n \binom{n-1}2 p^2 + \tfrac12 \binom n2 \binom{n-2}2  p^2(1-p^2)^2.
\end{align*} 
For $p=1/\sqrt{3}$ we get
\[
Ef_1=n^4/54 + n^3/18 - 8n^2/27 + n/3 \ge n^4/54.
\] 
From this it follows that $g(n) \ge n^4/54$.  
 
\noindent
(ii)  \emph{Upper bound} \\
Let $G$ be an arbitrary graph with $n$ vertices and $\rho \tfrac{n^2}{2}$ edges ($0< \rho < 1$).  
Since the cycle of length $4$ is isomorphic to the complete bipartite graph $K_{2,2}$, Lemma~\ref{bipartite subgraphs} 
shows that $c_4(G) \ge (\tfrac{1}{8}+o(1))\rho^4 n^4$. 
Furthermore, as each clique of size $4$ contains exactly three cycles of length $4$, we have $c_4(G) \ge 3k_4(G)$. 
Therefore, the number of edges of $\PP(G)$ is at most
\begin{align*}\textstyle
    \binom{e(G)}2  -2c_4(G) +3k_4(G) & \le\textstyle \binom{e(G)}2 -c_4(G)\\
        & \le \tfrac{1}{8} \rho^2 n^4 -\tfrac{1}{8}\rho^4 n^4 + o(n^4) \\
        & \le (\tfrac{1}{32}+o(1))n^4.
\end{align*}        
    \vskip-6.75mm%
\end{proof} 
\medskip

\begin{remark*}
The upper bound in Theorem~\ref{Bound the number of edges} is not tight:
  For any pair of graphs $F$ and $G$, let $N(F,G)$ denote the number of labeled copies of $F$ in $G$. 
A sequence $(G_n)$ of graphs ($|G_n| \rightarrow \infty)$ is \emph{quasi-random with density $p$} ($0<p<1$) if, for every graph $F$, 
\begin{align*} \label{quasi-random}
N(F,G_n)=(p^{e(F)}+o(1))|G_n|^{|F|}.
\end{align*}
According to Chung, Graham, and Wilson~\cite{CGW} this happens iff $N(K_2,G_n)=(p+o(1))|G_n|^2$ and $N(C_4,G_n)=(p^4+o(1))|G_n|^4$. 

If the upper bound of Theorem~\ref{Bound the number of edges} is tight, then we can find a sequence $(G_n)$ of graphs such that $e(G_n)=(\tfrac{1}{\sqrt{8}}+o(1))|G_n|^2$,
$c_4(G_n)=(\tfrac{1}{32}+o(1))|G_n|^4$, and $k_4(G_n)=(\tfrac{1}{96}+o(1))|G_n|^4$. 
This forces that $N(K_2,G_n)=(\tfrac{1}{\sqrt{2}}+o(1))|G_n|^2$, $N(C_4,G_n)=(\tfrac{1}{4}+o(1))|G_n|^4$, and $N(K_4,G_n)=(\tfrac{1}{4}+o(1))|G_n|^4$. 
The first two equalities imply that $(G_n)$ is quasi-random with density $\tfrac{1}{\sqrt{2}}$. 
Hence $N(K_4,G_n)=(\tfrac{1}{8}+o(1))|G_n|^4$, a contradiction.

It remains to be explored whether the upper bound can be improved by formalizing the subgraph count via flag algebras as in~\cite{HNPS}. 
\end{remark*}

\section{Neighborly edge polytopes}\label{sec:3}

Here we provide a forbidden subgraph characterization for $k$-neighborly edge polytopes, and then determine the maximal number of vertices of such polytopes.
For this we first prepare some 
notation.

Given a family $\FF$ of graphs, a graph $G$ is \emph{$\FF$-free} if it contains no copy of a graph in $\FF$ as a subgraph.
The \emph{Tur\'{a}n number} $\ex{(n,\FF)}$ is the maximal number of edges in an $\FF$-free graph on $n$ vertices. 
The \emph{Zarankiewicz number} $\z(n,\FF)$ is the maximal number of edges in an $\FF$-free {bipartite} graph on $n$ vertices.

A polytope $P$ is $k$-\emph{neighborly} if every subset of at most $k$ of its vertices defines a face of $P$. Thus every polytope is $1$-neighborly, and a polytope is $2$-neighborly if and only if its graph is complete. Except for simplices, no $d$-dimensional polytope is more than $\lfloor \frac{d}{2}\rfloor$-neighborly. 

For $k \ge 2$, let $\FF_k$ be a family of graphs on at most $2k$ vertices consisting of
\begin{compactitem}[$\bullet$]
        \item even cycles,
        \item graphs obtained by joining two odd cycles by a path.
\end{compactitem}
For example, $\FF_2=\{C_4\}$ and $\FF_3=\{C_4,F_2,2K_3+e,C_6\}$ (see Figure 2).
 
\begin{pspicture}(0.00,-1.40)(13,3)
\pspolygon(1.80,0.80)(3.20,0.80)(3.20,2.20)(1.80,2.20)
\put(1.80,0.80){\circle*{0.15}}
\put(3.20,0.80){\circle*{0.15}}
\put(3.20,2.20){\circle*{0.15}}
\put(1.80,2.20){\circle*{0.15}}
\put(2.40,0.10){$C_4$}
\pspolygon(4.60,0.80)(5.50,1.50)(6.40,0.80)(6.40,2.20)(5.50,1.50)(4.60,2.20)
\put(4.60,0.80){\circle*{0.15}}
\put(5.50,1.50){\circle*{0.15}}
\put(6.40,0.80){\circle*{0.15}}
\put(6.40,2.20){\circle*{0.15}}
\put(4.60,2.20){\circle*{0.15}}
\put(5.30,0.10){$F_2$}
\pspolygon(7.80,0.80)(8.60,1.50)(9.20,1.50)(10.00,0.80)(10.00,2.20)(9.20,1.50)(8.60,1.50)(7.80,2.20)
\put(7.80,0.80){\circle*{0.15}}
\put(8.60,1.50){\circle*{0.15}}
\put(9.20,1.50){\circle*{0.15}}
\put(10.00,0.80){\circle*{0.15}}
\put(10.00,2.20){\circle*{0.15}}
\put(7.80,2.20){\circle*{0.15}}
\put(8.20,0.10){$2K_3+e$}
\pspolygon(11.40,1.50)(11.80,0.80)(12.60,0.80)(13.00,1.50)(12.60,2.20)(11.80,2.20)
\put(11.40,1.50){\circle*{0.15}}
\put(11.80,0.80){\circle*{0.15}}
\put(12.60,0.80){\circle*{0.15}}
\put(13.00,1.50){\circle*{0.15}}
\put(12.60,2.20){\circle*{0.15}}
\put(11.80,2.20){\circle*{0.15}}
\put(12.05,0.10){$C_6$}
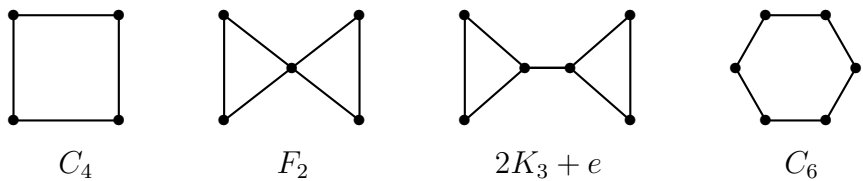
\put(5.20,-1){Figure 2: the family $\FF_3$.}
\end{pspicture}  
 
Now we can characterize $k$-neighborly edge polytopes for $k \ge 2$.

\begin{theorem} \label{thm:neighborly}  
For $k \geq 2$ the edge polytope $\PP(G)$ of a graph $G$ with at least $k$ edges is $k$-neighborly if and only if $G$ is
$\FF_k$-free.
\end{theorem}

Before proving Theorem~\ref{thm:neighborly} we state a consequence.
Let $\mathcal{C}_{2k}^{\rm even}$ denotes the family of all even cycles of lengths at most $2k$. As a bipartite graph is $\FF_k$-free if and only if it is $\mathcal{C}_{2k}^{\rm even}$-free, Theorem~\ref{thm:neighborly} implies the following result. 

\begin{corollary} \label{cor:neighborly_bipartite}
For $k \ge 2$ the edge polytope $\PP(G)$ of a bipartite graph $G$ with at least $k$ edges 
is $k$-neighborly if and only if $G$ is
$\mathcal{C}_{2k}^{\rm even}$-free.
\end{corollary} 

For the proof of Theorem~\ref{thm:neighborly}, we use the following lemma, which is a straightforward consequence of
results by Oshugi \& Hibi \cite[Lemmas 1.4 and 1.5]{OH}.
\begin{lemma}
\label{lem:simplex}
Let $H$ be a finite graph. Then $\PP(H)$ is a simplex if and only if every cycle in $H$ is odd, and every connected component of $H$ has at most one odd cycle.
\end{lemma}

This lemma can be reformulated in terms of forbidden subgraphs as follows.
\begin{lemma}
\label{lem:simplex_reformulation}
Suppose that $k \ge 2$ and $H$ is a graph with at most $2k$ vertices. Then $H$ is $\FF_k$-free if and only if every cycle in $H$ is odd, and every connected component of $H$ has at most one odd cycle.
\end{lemma}
\begin{proof}
The ``if'' part is obvious. For the ``only if'' part we observe that if a graph contains two odd cycles that
intersect in more than one vertex, then it also has an even cycle.
\end{proof}

We are now ready to prove Theorem~\ref{thm:neighborly}.
\begin{proof}[Proof of Theorem~\ref{thm:neighborly}]
Assume that $G$ is $\FF_k$-free. Let $\{i_1,j_1\},\ldots,\{i_k,j_k\}$ be $k$ different edges of $G$. Set $U=\{i_1,j_1,\ldots,i_k,j_k\}$, then $|U| \le 2k$. By Lemma~\ref{lem:simplex_reformulation} and Lemma~\ref{lem:simplex}, $\PP(G[U])$ is a simplex. 
Therefore, $\conv\{\e_{i_1j_1},\ldots,\e_{i_kj_k}\}$ is a face of $\PP(G)$ by Lemma~\ref{lem:local_property}. From this it follows that $\PP(G)$ is $k$-neighborly. 

Assume that $\PP(G)$ is $k$-neighborly. Let $U \in \binom{V(G)}{2k}$ be arbitrary. If we set
$\ell=\min\{k,e(G[U])\}$, then the number of non-isolated vertices of $G[U]$ is at most $2\ell$.
Lemma~\ref{lem:dimension} now shows that $\PP(G[U])$ has dimension at most $2\ell-1$. On the other hand, $\PP(G[U])$ is $\ell$-neighborly by Lemma~\ref{lem:local_property}. Hence $\PP(G[U])$ is a simplex. Lemmas~\ref{lem:simplex} and~\ref{lem:simplex_reformulation} imply that $G[U]$ is $\FF_k$-free for each $U \in \binom{V(G)}{2k}$. Thus, $G$ is $\FF_k$-free.
\end{proof}
 
 Theorem~\ref{thm:neighborly} can be used to obtain the following upper bound on the number of vertices of $k$-neighborly edge polytopes.
 \begin{corollary} \label{cor:turan_polytope}
 Let $k \ge 2$ be a fixed integer. Then a $k$-neighborly edge polytope of an $n$-vertex graph has at most 
 $\frac{1}{2}n^{1+1/k}+\left(\frac{k-1}{2k}+o(1)\right)n$ vertices.
 
 Furthermore, for each $k \in \{2,3,5\}$ there are infinitely many positive integers $n$ for which there is an $n$-vertex graph $G_n$ whose edge polytope is $k$-neighborly with at least $\frac{1}{2}n^{1+1/k} + \frac{k-1}{2k}n - n^{1-1/k}$ vertices.
 \end{corollary}

\begin{proof}
In the following we will use  
results on Tur\'{a}n numbers and on pseudorandomness,  
Theorems~\ref{thm:turan_number} and~\ref{thm:pseudorandomness}, which are presented in the appendix.

Since $\PP(G)$ is $k$-neighborly, Theorem~\ref{thm:neighborly} implies that the graph $G$ is 
$\FF_k$-free. Under this condition we will show that, as $n$ goes to infinity,
$e(G) \le \frac{1}{2}n^{1+1/k}+(\tfrac{k-1}{2k}+o(1))n$.
We may assume that
$e(G) \ge \tfrac{1}{2}n^{1+1/k}$.
Now let $\mathcal{T}$ be the set of all odd cycles of length at most~$k$ in $G$, and let $U$ be the set of all vertices in $G$ which is contained in some element of $\mathcal{T}$. 
Because $G$ is $\FF_k$-free, elements in $\mathcal{T}$ are pairwise disjoint, and consequently $|\mathcal{T}| \le n/3$. 
By removing one edge from each element of $\mathcal{T}$ we get a 
$(\mathcal{C}_{2k}^{\rm even} \cup \mathcal{C}_k)$-free graph $H$ with 
$e(H) \ge \frac{1}{2}n^{1+1/k}-\frac{1}{3}n$, where $\mathcal{C}_k$ denotes the family of all cycles of lengths at most~$k$. 
Since $H$ is $(\mathcal{C}_{2k}^{\rm even} \cup \mathcal{C}_k)$-free,
Theorem~\ref{thm:turan_number} tells us that
$e(H) \le \frac{1}{2}n^{1+1/k}+\frac{k-1}{2k}n+n^{1-1/k}$. Therefore, $H$ has average degree $d_H \sim n^{1/k}$. Theorem~\ref{thm:pseudorandomness} can be applied showing that
\[
e_{H}(S,T)=n^{-1+1/k}|S||T|+o(n^{1+1/k}) \quad \textrm{for every $S,T \subseteq V(G)$}.
\]
Since $H$ is obtained from $G$ by deleting $o(n^{1+1/k})$ edges, a similar formula holds for~$G$, namely
\[
e_G(S,T)=n^{-1+1/k}|S||T| + o(n^{1+1/k}) \quad \textrm{for every $S,T \subseteq V(G)$. }
\]
As $G$ is $\FF_k$-free, the induced subgraph $G[U]$ is the disjoint union of elements in~$\mathcal{T}$.  Hence $e_G(U,U)=2|U|\ge 6|\mathcal{T}|$.  On the other hand, $e_G(U,U)=n^{-1+1/k}|U|^2+o(n^{1+1/k})$. 
It follows that $|U|=o(n)$, and so $|\mathcal{T}|=o(n)$. 
From  
Theorem~\ref{thm:turan_number} we obtain
\[
e(G) \le \ex(n,\mathcal{C}_{2k}^{\rm even} \cup \mathcal{C}_k)+|\mathcal{T}| \le \tfrac{1}{2}n^{1+1/k}+(\tfrac{k-1}{2k}+o(1))n.
\]
 
 According to Theorem~\ref{thm:turan_number}, for infinitely many positive integers $n$ there is a $(\mathcal{C}_{2k}^{\rm even} \cup \mathcal{C}_k)$-free graph $G_n$ on $n$ vertices such that $e(G_n) \ge \tfrac{1}{2}n^{1+1/k}+ \tfrac{k-1}{2k}n - n^{1-1/k}$.  By Theorem~\ref{thm:neighborly}, the edge polytopes of these graphs have the desired property.
\end{proof}  

We also have the following lower bound on the maximal number of vertices of a\break $k$-neighborly edge polytope.
\begin{corollary}\label{cor:turan_polytope_bipartite}
 Let $k \ge 2$ be a fixed integer. Then the following holds. 
 \begin{compactenum}[\rm (i)]
 \item A $k$-neighborly edge polytope of an $n$-vertex graph has at most $\tfrac{1}{2}n^{1+1/k}+ O(n)$ vertices. Moreover, there are graphs $G_n$ on $n$ vertices such that $\PP(G_n)$ is a $k$-neighborly polytope with $\Omega(n^{1+2/3k})$ vertices. 
 \item A $k$-neighborly edge polytope of an $n$-vertex bipartite graph has at most $(\frac{1}{2})^{1+1/k} n^{1+1/k}$\break $+\, O(n)$ vertices. Moreover, for $k \in \{2,3,5\}$ this bound is tight up to the linear term for infinitely many $n.$  
 \end{compactenum}
\end{corollary}

\begin{proof} {\rm (i)} If the edge polytope $\PP(G)$ is $k$-neighborly, then $G$ is $\mathcal{C}_{2k}^{\rm even}$-free by Theorem~\ref{thm:neighborly}. Thus $\PP(G)$ has at most 
$\ex(n,\mathcal{C}_{2k}^{\rm even}) \le \frac{1}{2}n^{1+1/k} + O(n)$ vertices, according to Lam \& Verstra\"ete \cite{LV}.
 On the other hand, Sarnak et al. \cite{LPS} showed that 
 $\ex(n,\{C_3,C_4,\ldots,C_{2k}\})=\Omega\left(n^{1+2/3k}\right)$.
 Thus Theorem~\ref{thm:neighborly} completes the proof.
 
 {\rm (ii)} If the edge polytope $\PP(G)$ of a bipartite graph $G$ is $k$-neighborly, then the graph is
 $\mathcal{C}_{2k}^{\rm even}$-free by Theorem~\ref{thm:neighborly}.
 It follows from \cite{Hoory} that the polytope $\PP(G)$ has at most
 $\z(n,\mathcal{C}_{2k}^{\rm even}) \le \left(\frac{1}{2}\right)^{1+1/k} n^{1+1/k} + O(n)$ vertices. 
 
 For $k \in \{2,3,5\}$ the existence of  
 \emph{generalized polygons} (Erd\H{o}s \& R\'enyi \cite{ER}, Benson \cite{Ben}, and Singleton \cite{Sing}) shows that
 $\z(n,\mathcal{C}_{2k}^{\rm even}) \ge \left(\frac{1}{2}\right)^{1+1/k} n^{1+1/k} + O(n)$
 for infinitely many $n$. Thus Corollary~\ref{cor:neighborly_bipartite} completes the proof.
 \end{proof}

\section{Edge polytopes with many facets}\label{sec:4}

In this section we study the maximal number of facets of a $d$-dimensional edge polytope.
  Here we only deal with edge polytopes of {connected graphs}; all results can easily be extended to the general case. 
  
  We use some terminology of Ohsugi \& Hibi \cite{OH}, as follows. 
  Let $G$ be a connected graph on the vertex set $[n]$.  A vertex $i \in [n]$ is \emph{regular} (resp.\ \emph{ordinary}) in $G$ if $G[[n] {\setminus}i]$ has no bipartite components (resp.\ if $G[[n] {\setminus} \{i\}]$ is connected). \\
  A subset $\emptyset \neq A \subseteq [n]$ is \emph{independent} in $G$ if $N(\{i\}) \cap A =\emptyset$ for all $i \in A$. 
  If $A$ is independent in $G$, then the bipartite graph induced by $A$ in $G$ is defined to be the graph having the vertex set $A \cup N(A)$ and consisting of all edges $\{i,j\}$ of $G$ with $i \in A$ and $j \in N(A)$.  This graph will be denoted by $G[A,N(A)]$.  \\
  When $G$ is non-bipartite, we say that a subset $\emptyset \neq A \subseteq [n]$ is \emph{fundamental} in~$G$~if%
  \begin{compactenum}[(i)]
    \item  $A$ is independent in $G$ and $G[A,N(A)]$ is connected, and 
    \item $G[[n]{\setminus} (A \cup N(A))]$ has no bipartite components. 
\end{compactenum}
  When $G$ is bipartite, a subset $\emptyset \neq A \subseteq [n]$ is \emph{acceptable} in $G$ if
  \begin{compactenum}[(i)]
    \item $A$ is independent in $G$ and $G[A,N(A)]$ is connected, and  
   \item $G[[n] {\setminus} (A \cup N(A))]$ is a connected graph with at least one edge.  
\end{compactenum}
  Let $A$ be an non-empty \emph{independent set} of $G$. 
  Denoted by $\mathcal{H}_{A}^{+}$ the closed half-space
\[
\mathcal{H}_{A}^{+}=\{\x \in \RR^n: \sum\limits_{i \in N(A)}x_i-\sum\limits_{j\in A}x_j \ge 0 \},
\]
and by $\mathcal{H}_{A}$ the hyperplane
\[
\mathcal{H}_{A}=\{\x \in \RR^n: \sum\limits_{i \in N(A)}x_i-\sum\limits_{j\in A}x_j = 0 \}.
\]
  If $i \in [n]$, then we write $\mathcal{H}_i^{+}$ for the closed half-space
\[\mathcal{H}_i^{+}=\{\x \in \RR^n: x_i \ge 0 \},\]
  and $\mathcal{H}_i$ for the hyperplane
\[\mathcal{H}_i=\{\x \in \RR^n: x_i = 0 \}.\]
 
\begin{theorem}  [{Ohsugi \& Hibi \cite[Theorem 1.7]{OH}}] \label{thm:the_facets}~
\begin{compactenum}[\rm(i)]
 \item  Let $G$ be a connected non-bipartite graph on the vertex set $[n]$.
  Let $\Psi$ denote the set of those hyperplanes $\mathcal{H}_i$ such that $i$ is regular in $G$ and of those hyperplanes $\mathcal{H}_{A}$ such that $A$ is fundamental in $G$.  Then the set of facets of the edge polytope $\PP(G)$ is $\{\mathcal{H} \cap \PP(G): \mathcal{H} \in \Psi \}$.  
 \item Let $G$ be a connected bipartite graph on the vertex set $V(G)=[n]$, and let $V(G)=V_1 \cup V_2$ be the partition of $V(G)$. 
  Let $\Psi$ denote the set of those hyperplanes $\mathcal{H}_i$ such that $i$ is ordinary in $G$ and of those hyperplanes $\mathcal{H}_{A}$ such that $A$ is acceptable in $G$ with $A \subset V_1$. 
  Then the set of facets of the edge polytope $\PP(G)$ is $\{\mathcal{H} \cap \PP(G): \mathcal{H} \in \Psi \}$. 
\end{compactenum}
\end{theorem}
  
\noindent 
Let $d \in \NN$.  We write $f(d)$ for the maximal number of facets of $\PP(G)$, where $G$ ranges over all connected graph such that $\dim{(\PP(G))}=d$. 
 
  \begin{lemma} \label{upper bound}
  $f(d) \le 2^d+d$ for all $d \ge 3$. 
  \end{lemma}
  
  \begin{proof} Let $G$ be a connected graph on $[n]$ with $\dim{(\PP(G))}=d \ge 3$. 
  Denote by $f_{d-1}$ the number of facets of $\PP(G)$.  
  It is sufficient to prove that $f_{d-1} \le 2^d+d$. 
  We distinguish two cases. 
  
  If $G$ is  bipartite, then $d=n-2$. 
  Let $V_1 \cup V_2$ be the partition of $V(G)$.  
  We can assume that $|V_1| \le |V_2|$. 
  Applying Theorem~\ref{thm:the_facets}, we get $f_{d-1} \le 2^{|V_1|}+n \le 2^{\lfloor n/2 \rfloor}+n \le 2^{\lfloor (d+2)/2 \rfloor}+d+2 <2^d+d$. 
  
  If $G$ is   non-bipartite, then $d=n-1$. 
  We denote by $\FF$ the family of independent sets in $G$.  
  Let $A \subseteq [n]$.  Since $G$ is not bipartite, either $A \notin \FF$ or $A^{c} \notin \FF$. 
  It follows that $|\FF| \le 2^{n-1}-1$. 
  By Theorem~\ref{thm:the_facets}, we see that $f_{d-1} \le |\FF|+n \le (2^{n-1}-1)+n=2^d+d$. 
  \end{proof}
  
  \begin{lemma} \label{lower bound}
  $f(d) > 4^{\lfloor d/3 \rfloor}$ for all $d>0$. 
  \end{lemma}
  
  \begin{pspicture}(-0.25,0)(12,5)
  \pspolygon(7,3)(6.20,4.60)(5.40,4.60)(5.40,3.80)
  \psline(7,3)(5.40,4.60)
  \psline(5.40,3.80)(6.20,4.60)
  \pspolygon(7,3)(8.60,3.80)(8.60,4.60)(7.80,4.60)
  \psline(7,3)(8.60,4.60)
  \psline(8.60,3.80)(7.80,4.60)
  \pspolygon(7,3)(7.80,1.40)(8.60,1.40)(8.60,2.20)
  \psline(7,3)(8.60,1.40)
  \psline(8.60,2.20)(7.80,1.40)
  \pspolygon(7,3)(5.40,2.20)(5.40,1.40)(6.20,1.40)
  \psline(7,3)(5.40,1.40)
  \psline(5.40,2.20)(6.20,1.40)
  \put(7,3){\circle*{0.15}}
  \put(5.43,2.88){$3k+1$}
  \put(5.40,3.80){\circle*{0.15}}
  \put(5.00,3.60){$1$}
  \put(6.20,4.60){\circle*{0.15}}
  \put(6.40,4.60){$3$}
  \put(5.40,4.60){\circle*{0.15}}
  \put(5.00,4.60){$2$}
  \put(8.60,3.80){\circle*{0.15}}
  \put(8.80,3.60){$6$}
  \put(8.60,4.60){\circle*{0.15}}
  \put(8.80,4.60){$5$}
  \put(7.80,4.60){\circle*{0.15}}
  \put(7.40,4.60){$4$}
  \put(7.80,1.40){\circle*{0.15}}
  \put(7.50,0.95){$3i$}
  \put(8.60,1.40){\circle*{0.15}}
  \put(8.50,0.95){$3i-1$}
  \put(8.60,2.20){\circle*{0.15}}
  \put(8.50,2.40){$3i-2$}
  \put(5.40,2.20){\circle*{0.15}}
  \put(4.80,2.20){$3k$}
  \put(5.40,1.40){\circle*{0.15}}
  \put(4.30,0.95){$3k-1$}
  \put(6.20,1.40){\circle*{0.15}}
  \put(6.00,0.95){$3k-2$}
  \put(3.30,0){Figure 3: The windmill graph $\windmill{(4,k)}$. }
  \end{pspicture}
   
  \begin{proof}
  Without restriction we can assume that $d=3k$.  Let $G$ be the windmill graph $\windmill{(4,k)}$
  on the vertex set $[3k+1]$ with the edge set
  \begin{align*}
  E(G)= & \big \{ \{j,3k+1\} : j=1,\ldots,3k \} \ \cup \\
        & \big \{ \{3i-2,3i-1\},\{3i-2,3i\},\{3i-1,3i\}: i=1,\ldots,k \big \}.
  \end{align*}
  As $G$ is a connected non-bipartite graph, we have $\dim{(\PP(G))}=(3k+1)-1=d$. 
  We will now determine all fundamental sets in $G$. 
  Observe that a non-empty subset $A \subseteq [3k+1]$ is independent in $G$ if and only if 
  \begin{compactenum}[(i)]
  \item $A=\{3k+1\}$, or 
  \item $3k+1 \notin A$, and $|A \cap \{3i-2,3i-1,3i\}| \le 1$ for all $i=1,\ldots,k$.  
  \end{compactenum}
  We claim that such a set $A$ is fundamental in $G$.  
  There are two possible cases. 
  If $A=\{3k+1\}$, then $G[A,N(A)]$ is the graph with vertex set $[3k+1]$ and edge set
  $\big \{ \{j,3k+1\}: j=1,\ldots,3k \big \}$.  
  Since $A \cup N(A)=[3k+1]$ and $G[A,N(A)]$ is connected, we see that $A$ is fundamental in~$G$.  
  If $3k+1 \notin A$ and $|A \cap \{3i-2,3i-1,3i\}| \le 1$ for $i=\overline{1,k}$, then we can assume that $A=\{3,6,\ldots,3\ell\}$ for some $\ell \le k$. 
  In this case, we have
  \[ G[A,N(A)]=\big \{\{3s,3s-2\},\{3s,3s-1\},\{3s,3k+1\} : s=1,\ldots,\ell \big \}, \]
  \[ G[\overline{A \cup N(A)}]=\big \{ \{3t-2,3t-1\},\{3t-2,3t\}, \{3t-1,3t\} : t=\ell+1,\ldots,k \big \}. \]
  Since $G[A,N(A)]$ is connected and $G[\overline{A \cup N(A)}]$ has no bipartite components, $A$ is fundamental in $G$.  
  
  We next claim that $i \in [3k+1]$ is regular in $G$ if and only if $i \neq 3k+1$.  Indeed, we distinguish two cases.
  If $i=3k+1$, then the induced subgraph $G[[3k+1] {\setminus}i]$ is a disjoint union of $k$ triangles. It follows that $G[[3k+1] {\setminus} i]$ is not connected, and consequently $i$ is not regular in $G$.  \\
  If $i \neq 3k+1$, then we can assume that $i=1$.  In this case, $G[[3k+1] {\setminus}i]$ is the graph with vertex set $\{2,3,\ldots,3k+1\}$ and edge set
  \begin{align*}
  E=& \big \{ \{j,3k+1\} : j=4,\ldots,3k  \big \} \ \cup \\
    & \big \{ \{2,3\},\{2,3k+1\},\{3,3k+1\} \big \} \ \cup \\ 
    & \big \{ \{3i-2,3i-1\},\{3i-2,3i\},\{3i-1,3i\}: i=2,\ldots,k \big \}.
  \end{align*}
  We can see that this graph is connected and non-bipartite. 
  Hence $i$ is a regular vertex in~$G$, as desired. 
  Finally, applying Theorem~\ref{thm:the_facets} we conclude that the number of facets of $\PP(G)$ is $\sum_{\ell=1}^k \binom k \ell 3^{\ell} + 3k=4^k+3k-1$. 
  Therefore, $f(d) > 4^k = 4^{\lfloor d/3 \rfloor}$.
  \end{proof}
 
  As a consequence of these lemmas, we obtain the following bounds on~$f(d)$.
  
  \begin{theorem}\label{thm:exponential_facets}
  For all $d \ge 3$, the maximal number $f(d)$ of facets of a $d$-dimensional edge polytope satisfies
\[4^{\lfloor d/3 \rfloor} < f(d) \le 2^d+d.\]
  \end{theorem}
 
 
 \section*{Appendix: Tur\'{a}n numbers}
 Here we give a tight asymptotic upper bound on the Tur\'an number $\ex(n,\mathcal{C}_{2k}^{\rm even} \cup \mathcal{C}_k)$, where $\mathcal{C}_{2k}^{\rm even}=\{C_4,C_6,\ldots,C_{2k}\}$ and $\mathcal{C}_k=\{C_3,C_4,\ldots,C_k\}$. We also show that every nearly extremal $(\mathcal{C}_{2k}^{\rm even}\cup \mathcal{C}_k)$-free graph is pseudorandom.
 
 The basic estimates for Tur\'{a}n numbers  
 for even cycles are obtained by counting various types of walks in graphs: 
 A \emph{non-returning walk of length} $k$ in $G$ is a sequence $v_0e_0v_1e_1 \ldots v_{k-1}e_{k-1}v_k$ such that $v_i \in V(G)$, $e_i=\{v_i,v_{i+1}\} \in E(G)$, and $e_i \neq e_{i+1}$ for $0 \le i <k$. 
 Let $\nu_k(G)$ denote the average number of non-returning walks of length $k$ in~$G$. If $G$ is a $d$-regular graph on $n$ vertices then clearly $\nu_k(G)=d(d-1)^{k-1}$. For irregular graphs, we have the following lower bound. 
 
 \begin{proposition}[Alon, Hoory \& Linial \cite{AHL}] \label{prop:moore_bound}
 If $G$ is a graph with minimum degree at least~$2$ and average degree $d$, then $\nu_k(G) \ge d(d-1)^{k-1}$. 
 \end{proposition}
 
 The following simple result will be very useful for our investigation. It is probably well-known, but we couldn't find a reference for it.
 
 \begin{lemma}
 \label{lem:union_paths}
 Suppose that $P$ and $Q$ are two different paths of length $k \ge 2$ with the same endpoints. If $P \cup Q$ is $\mathcal{C}_k$-free, then $P=\alpha P'\beta$ and $Q=\alpha Q'\beta$ for some vertex-disjoint paths $\alpha,\beta,P'$ and $Q'$.
 \end{lemma}
 
 \begin{proof}
 The lemma is obviously true for $k \in \{2,3\}$. So let $k>3$ and proceed by induction. Let $x$ and $y$ be endpoints of $P$ and $Q$. We distinguish three cases.
 
 \noindent{\emph{Case 1:} ${N_P(y)=N_Q(y):=v}$}.\\
 Let $P_1$ and $Q_1$ be the subpaths of $P$ and $Q$ from $x$ to $v$, respectively. Then they are two different paths of length $k-1$ from $x$ to $v$, and their union is $\mathcal{C}_k$-free. By induction, $P_1=\alpha P'\beta$ and $Q_1=\alpha Q'\beta$ for some vertex-disjoint paths $\alpha,\beta,P',Q'$ with $x \in \alpha$ and $v \in \beta$. Hence $P=\alpha P'\beta y$ and $Q=\alpha Q' \beta y$. 
 
 \noindent{\emph{Case 2:} ${N_P(y)\neq N_Q(y) \ \text{and} \ N_P(y) \in Q}$}.\\
 In this case, $N_P(y)Qy$ is a cycle of length at most $k$. This contradicts the assumption that $P \cup Q$ is $\mathcal{C}_k$-free.
 
 \noindent{\emph{Case 3:} ${N_P(x),N_P(y)\notin Q \ \text{and} \ N_Q(x),N_Q(y) \notin P}$}.\\
 We identify $x,N_P(x)$ and $N_Q(y)$ (resp.~$y,N_P(y)$ and~$N_Q(y)$) as a new vertex $x'$ (resp.~$y'$). Let $P_2$ and $Q_2$ be the new paths corresponding to $P$ and $Q$. Then they are two different paths of length $k-2$ from $x'$ to $y'$. Since $P \cup Q$ is $\mathcal{C}_k$-free, $P_2 \cup Q_2$ is $\mathcal{C}_{k-2}$-free. By induction $P_2=\alpha' P''\beta'$ and $Q_2=\alpha'Q''\beta'$ for some disjoint paths $\alpha',\beta',P''$ and $Q''$. We claim that $x'$ is the only vertex of $\alpha'$. Otherwise, $\alpha'=x'u\alpha''$ for some vertex $u$ and path~$\alpha''$. In this case, the vertices $x,N_P(x),u$ and $N_Q(x)$ form a cycle of length $4$ in $P \cup Q$, which contradicts the fact that $P \cup Q$ is $\mathcal{C}_k$-free. Similarly, $y'$ is the only vertex of $\beta'$. From this it follow that $P$ and $Q$ are two internally vertex-disjoint paths. Consequently, they have the desired structures.
 \end{proof}
 
 Using the previous lemmas we can determine the Tur\'an number of the family $\mathcal{C}_{2k}^{\rm even} \cup \mathcal{C}_k$ up to a sublinear term.
 
 \begin{theorem} \label{thm:turan_number}
 Suppose $k\geq 2$ and $n \geq 1$. Then
 \[\ex{(n,\mathcal{C}_{2k}^{\rm even} \cup \mathcal{C}_k)} \le \tfrac{1}{2}n^{1+1/k} +\tfrac{k-1}{2k}n+n^{1-1/k}.\]
 Furthermore, if $k \in \{2,3,5\}$, then for infinitely many $n$
 \[\ex{(n,\mathcal{C}_{2k}^{\rm even} \cup \mathcal{C}_k)} \ge \tfrac{1}{2}n^{1+1/k} +\tfrac{k-1}{2k}n-n^{1-1/k}.\]
 \end{theorem}
 
 It is quite well known (K\H{o}v\'{a}ri et al.~\cite{KST} and Reiman \cite{Reiman}, cf.~\cite[Chap.~25]{Z58-4}) that $\ex(n,C_4)\le \lfloor \frac{n}{4}(1+\sqrt{4n-3})\rfloor$.
  On the other hand, Erd\H{o}s \& R\'{e}nyi \cite{ER} and Brown \cite{Brown} proved that for infinitely many positive integer~$n$ there is a $C_4$-free graph on $n$ vertices with $\frac{n-1}{4}(1+\sqrt{4n-3})$ edges. These results establish Theorem~\ref{thm:turan_number} in the case when $k=2$. It remains to handle the case when $k \ge 3$.
  
 \begin{proof}[Proof of Theorem~\ref{thm:turan_number}]
 As discussed above, it is enough to prove the theorem for $k \ge 3$. Let $G$ be a minimal counterexample to the theorem.
 Then $G$ has minimum degree at least~$2$ and average degree $d > n^{1/k} + \tfrac{k-1}{
 k}+2n^{-1/k}$. 
 We denote by $\PP_k$ the family of paths of length $k$ in $G$.  
 Since $G$ is $\mathcal{C}_k$-free, every non-returning walk of length $k$ is nothing but a path of length $k$.  
 It now follows from Proposition~\ref{prop:moore_bound} that $|\PP_k| \ge nd(d-1)^{k-1} > n^2$. 
 By Lemma~\ref{lem:union_paths}, for any ordered pair of vertices $u,v$ there is at most one path of length $k$ from $u$ to $v$. Thus, the number of paths of length $k$ is at most $n^2$, a contradiction.
  
 Let $\alpha \in \NN$. Set $q=2^{2\alpha +1}$ if $k=3$, and $q=3^{2\alpha+1}$ if $k=5$. Lazebnik et al. \cite{LUW} constructed a $(\mathcal{C}_{2k}^{\rm even} \cup \mathcal{C}_k)$-free graph $G$ on $n=q^k+\ldots + q+1$ vertices with $\frac{1}{2}\big\{(q+1)(q^k+\ldots +q+1)-(q^{\lfloor \frac{k+1}{2}\rfloor}+1)\big\}$ edges.
 We can verify that $e(G) \ge \frac{1}{2}n^{1+1/k} +\frac{k-1}{2k}n - n^{1-1/k}$ for large $n$. 
 \end{proof} 
 
 Another key ingredient in the proof of Corollary~\ref{cor:turan_polytope} is the notion of \emph{pseudorandomness}. We refer the reader to 
 Krivelevich \& Sudakov \cite{KS} for a survey.  
 The following result expresses the pseudorandomness property of a nearly extremal $(\mathcal{C}_{2k}^{\rm even} \cup \mathcal{C}_k)$-free graph: 
 For any two large sets the number of ordered edges between them is close to what one would expect in a random graph of the same edge density.

 \begin{theorem} \label{thm:pseudorandomness}
 Let $k \ge 2$ be a fixed integer. Suppose $G$ is a $(\mathcal{C}_{2k}^{\rm even} \cup \mathcal{C}_k)$-free graph on $n$ vertices with average degree $d \sim n^{1/k}$. 
 Then $e_G(S,T)= \frac{d}{n}|S||T|+o(n^{1+1/k})$ for any $S,T \subseteq V(G)$. 
 \end{theorem}

 \begin{proof}[Sketch] 
 Our proof follows the lines of a remark of Keevash et al.\ \cite[Section~9]{KSV}. We just sketch the argument, and refer to \cite{KSV} for the omitted details.   
 Suppose that $G$ has eigenvalues $\lambda_1 \ge \lambda_2 \ge \cdots \ge \lambda_n$. 
 Let $w_{2k+2}^{\circ}(G)$ denote the number of closed walks of length $2k+2$ in $G$ divided by $n$.  
 Since $G$ is $(\mathcal{C}_{2k}^{\rm even} \cup \mathcal{C}_k)$-free, Lemma~\ref{lem:union_paths} shows that there is at most one path of length $k$ between any pair of vertices of $G.$ Using this property we control the maximum degree as $\Delta < (1+\varepsilon)d$ by deleting $o(n^{1+1/k})$ edges. 
 Then the argument of \cite[Lemma 3.4]{KSV} shows $w_{2k+2}^{\circ}(G)<(1+o(1))n\Delta^2;$ the only difference is that there are $n-1$ choices for $u$ rather than $n/2+o(n)$.  
 On the other hand, $w_{2k+2}^{\circ}(G)=\frac{1}{n}\sum \lambda_i^{2k+2}$ has a contribution of $d^{2k+2}/n \sim nd^2$ from the first eigenvalue, so the other eigenvalues are $o(d)$ as $\varepsilon \rightarrow 0$. 
 The pseudorandomness property now follows from the non-bipartite version of \cite[Lemma 5.4]{KSV}, which is provided in \cite[Section 2.4]{KS}.
 \end{proof}
 
 \subsubsection*{Acknowledgements.} We are grateful to Raman Sanyal for helpful discussions, and Tibor Szab\'o for drawing our attention to the paper by Keevash, Sudakov $\&$ Verstra\"ete \cite{KSV}. We would like to thank an anonymous referee for very instructive comments and suggestions.
 

\end{document}